\documentclass[a4paper]{article}

\usepackage{amssymb,amsmath, amsfonts}
\usepackage{authblk}
\usepackage{pst-grad} 

\usepackage{tikz-cd}
\usetikzlibrary{matrix,arrows,decorations.pathmorphing}

\usepackage{cmdtrack}
\usepackage{amsthm}

\usepackage{bussproofs}
\EnableBpAbbreviations

\usepackage{url}

\newcommand\mj{\mbox{\bf 1}}
\newcommand\set{\mbox{\emph{Set}}}
\newcommand\rel{\mbox{\emph{Rel}}}
\newcommand\seto{\set_\omega}
\newcommand\relo{\rel_\omega}
\newcommand\cat{\mbox{\emph{Cat}}}
\newcommand\catp{\cat_+}
\newcommand\catt{\cat_\times}

\newcommand\cirk{\,{\raisebox{.3ex}{\tiny $\circ$}}\,}

 \newtheorem{thm}{Theorem}[section]
 \newtheorem{prop}[thm]{Proposition}
 \newtheorem{lem}[thm]{Lemma}
 
 \newtheorem{exm}[thm]{Example}

\numberwithin{equation}{section}

\newcommand\pl{\!+\!}
\newcommand\mn{\!-\!}

\begin{document}

\title{Representing Conjunctive Deductions\\ by Disjunctive Deductions}
\author{Kosta Do\v sen and Zoran Petri\'c}
\affil{Mathematical Institute SANU,\\ Knez Mihailova 36, p.f.\
367,\\ 11001 Belgrade, Serbia\\
\texttt{kosta@mi.sanu.ac.rs}, \texttt{zpetric@mi.sanu.ac.rs} }

\date{}
\maketitle

\vspace{-3ex}

\begin{abstract}
\noindent A skeleton of the category with finite coproducts
$\mathcal{D}$ freely generated by a single object has a
subcategory isomorphic to a skeleton of the category with finite
products $\mathcal{C}$ freely generated by a countable set of
objects. As a consequence, we obtain that $\mathcal{D}$ has a
subcategory equivalent with $\mathcal{C}$. From a
proof-theoretical point of view, this means that up to some
identifications of formulae the deductions of pure conjunctive
logic with a countable set of propositional letters can be
represented by deductions in pure disjunctive logic with just one
propositional letter. By taking opposite categories, one can
replace coproduct by product, i.e. disjunction by conjunction, and
the other way round, to obtain the dual results.

\end{abstract}

\vspace{.3cm}

\noindent {\small {\it Mathematics Subject Classification} ({\it
2010}): ({\it
2010}): 03F03 (Proof theory, general), 03F07 (Structure of proofs), 03G30 (Categorial logic, topoi), 18A15 (Foundations, relations to logic and deductive systems) }

\vspace{.5ex}

\noindent {\small {\it Keywords$\,$}: conjunction, disjunction, deductions, categories with finite products and coproducts, Brauerian representation, exponential functor, contravariant power-set functor}

\vspace{.5ex}

\noindent {\small {\it Acknowledgements$\,$}: This work was
supported by project ON174026 of the Ministry of Education,
Science, and Technological Development of the Republic of Serbia.
}

\section{Introduction}

In general proof theory, as conceived by \cite{P71}, one addresses the question ``What is a proof?'', or ``What is a deduction?''---a deduction being a hypothetical proof, i.e.\ a proof with hypotheses---by dealing with questions related to normal forms for deductions, and in particular with the question of identity criteria for deductions. This theory deals with the structure of deductions, which in one way can be shown with the help of the typed lambda calculus in the Curry-Howard correspondence, and not with their strength measured by ordinals, which is what one finds in proof theory that arose out of Hilbert's programme.

Much of general proof theory is the field of categorial (or categorical) proof theory. Fundamental notions of category theory like the notion of adjoint functor, and very important structures like cartesian closed categories, came to be of central concern for logic in that field. Through results of categorial proof theory called coherence results, which provide a model theory for equality of deductions, logic finds new ties with geometry, topology and algebra (see the books \cite{D99}, \cite{DP04} and \cite{DP07}, the more recent introductory survey \cite{D14}, and references therein).

In general proof theory, and in particular in categorial proof theory, one deals with an algebra of deductions, and for that, one concentrates on the operations of this algebra, which come with the rules of inference. As an equational theory, the algebra of deductions involves the question of identity criteria for deductions, the central question of general proof theory. (This question may be found, at least implicitly, in Hilbert's 24th problem; see \cite{D15}.)

In categorial proof theory one usually studies a freely generated category of a certain kind equationally presented. This freely generated category is constructed out of syntactical material, as in universal algebra one constructs a freely generated algebra of a certain kind equationally presented by factoring through an equivalence relation on terms. In categories we have partial algebras---the arrow terms out of which the equivalence classes are built have types, their sources and targets---but there is no significant mathematical difference in the construction when compared with what one has in universal algebra without types (see \cite{DP04}, Chapter~2, in particular in Section 2.3). The objects of this freely constructed categories are propositions, i.e.\ formulae, and the arrows, i.e.\ the equivalence classes of arrow terms, are deductions, i.e.\ equivalence classes of particular derivations, whose sources are premises and whose targets are conclusions. (Our dealing only with deductions with not more than one premise will not limit the mathematical importance of the results we consider, because once conjunction becomes available, a finite number of premises can always be replaced by their conjunction, and zero premises can be replaced by the propositional constant $\top$.) For deductions we have the partial operation of composition and identity deductions (this is essential for them; see \cite{D16}). The categories in question are interesting if they are not preorders, i.e., not all arrows with the same source and the same target are equal. Otherwise, the proof theory is trivial: any deductions with the same premises and the same conclusions become equal.

We deal with pure disjunctive logic with the alphabet including
just the symbol $\vee$ for the connective of disjunction, the propositional constant $\bot$ and a single propositional letter. The proof
theory of this fragment, according to Prawitz's equivalence of
derivations in its beta-eta version (see \cite{P71}) and Lawvere's
and Lambek's ideas about identity of deductions (see \cite{Law69} and \cite{L72}), is to be
identified with the category with finite, possibly empty, coproducts freely
generated by a single object. In this proof theory we have only hypothetical proofs, i.e.\ we have no proofs without hypotheses. We also consider pure
conjunctive logic based on the alphabet including just the symbol
$\wedge$ for the connective of conjunction, the propositional constant $\top$ and a countably infinite set of propositional letters.
The proof theory of that fragment is to be identified with the
category with finite, possibly empty, products freely generated by a countable set
of objects.

What one thinks of immediately when one has to connect conjunction
with disjunction is to put everything upside down, i.e.\ dualize.
The question whether without dualizing disjunction can
imitate conjunction, or the other way round, is however different, and the
answer to this question is not straightforward. The goal of this
paper is to give an answer to it in the context of categorial proof
theory. Our result that, up to some
identifications of formulae, the deductions of pure conjunctive
logic with a countable set of propositional letters can be
represented by deductions in pure disjunctive logic with just one
propositional letter, without dualizing, is rather unexpected.

To achieve this result, the propositional letters of conjunctive logic are coded by disjunctive formulae with a single propositional letter. So propositional variables are not translated by propositional variables.

We will show that there is an embedding, i.e.\ faithful functor,
which is one-one on objects from a category equivalent to the
later category to a category equivalent to the former category.
The main ingredient of our proof is the faithfulness of a
structure-preserving functor from a free cartesian category to the
category $\set$. \v Cubri\' c proved in \cite{Cu98} that there
exists a faithful structure-preserving functor from a free
cartesian closed category to the category \set. Our result may be
taken as related to that previous result, but it cannot be inferred from it.

In our proofs we rely on coherence results and results involving the representation theory of algebras called Brauerian algebras we have obtained before, for which we give references at appropriate places in the text. The results about representations of Brauerian algebras are however generalized. All these results enabled us to get new results about faithful embeddings, mentioned in the preceding paragraph, in a simple manner. We think that without proceeding in a manner such as ours, by relying on our previous coherence and representation results, what is achieved in this paper would be quite difficult to reach.

\section{The category $\seto$ and its coproducts}
Let $\omega$ be the set of finite ordinals $0,1,\ldots ,n,\ldots$,
i.e.\  $\emptyset, \{0\}, \ldots, \{0,\ldots,n-1\},\ldots$, and
let $\seto$ be the full subcategory of $\set$ whose set of objects
is $\omega$. This category is a strict monoidal category with
finite coproducts freely generated by the single object 1. This
can be demonstrated as follows.

The empty coproduct in $\seto$ is the object 0 and the binary
coproduct on objects is given by addition. The binary coproduct on
arrows is given by putting ``side by side'', i.e.,\ the coproduct
of $f\!:n\rightarrow m$ and $f'\!:n'\rightarrow m'$ is given by the function
$g\!:n+n'\rightarrow m+m'$ such that
\[
g(i)=\left\{
\begin{array}{cl}
f(i) & \mbox{if}\hspace{1em} 0\leq i\leq n-1,
\\[1ex]
m+f'(i-n) & \mbox{if}\hspace{1em} n\leq i\leq n+n'-1.
\end{array}
\right .
\]
For example, if $f\!:2\vdash 3$ and $f'\!:3\vdash 2$ are
given, respectively, by the following two pictures (representing functions going downwards, with sources, i.e.\ domains, above, and targets, i.e.\ codomains, below)
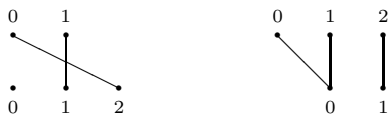
\begin{center}
\begin{picture}(140,50)
\put(0,15){\circle*{2}} \put(20,15){\circle*{2}}
\put(40,15){\circle*{2}}

\put(0,10){\makebox(0,0)[t]{\scriptsize $0$}}
\put(20,10){\makebox(0,0)[t]{\scriptsize $1$}}
\put(40,10){\makebox(0,0)[t]{\scriptsize $2$}}

\put(0,35){\circle*{2}} \put(20,35){\circle*{2}}

\put(0,40){\makebox(0,0)[b]{\scriptsize $0$}}
\put(20,40){\makebox(0,0)[b]{\scriptsize $1$}}

\put(0,35){\line(2,-1){40}} \put(20,35){\line(0,-1){20}}

\put(120,15){\circle*{2}} \put(140,15){\circle*{2}}

\put(120,10){\makebox(0,0)[t]{\scriptsize $0$}}
\put(140,10){\makebox(0,0)[t]{\scriptsize $1$}}

\put(100,35){\circle*{2}} \put(120,35){\circle*{2}}
\put(140,35){\circle*{2}}

\put(100,40){\makebox(0,0)[b]{\scriptsize $0$}}
\put(120,40){\makebox(0,0)[b]{\scriptsize $1$}}
\put(140,40){\makebox(0,0)[b]{\scriptsize $2$}}

\put(100,35){\line(1,-1){20}} \put(120,35){\line(0,-1){20}}
\put(140,35){\line(0,-1){20}}

\end{picture}
\end{center}
then $f+ f'\!:5\rightarrow 5$ is given by
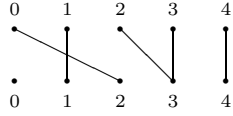
\begin{center}
\begin{picture}(80,50)
\put(0,15){\circle*{2}} \put(20,15){\circle*{2}}
\put(40,15){\circle*{2}}

\put(0,10){\makebox(0,0)[t]{\scriptsize $0$}}
\put(20,10){\makebox(0,0)[t]{\scriptsize $1$}}
\put(40,10){\makebox(0,0)[t]{\scriptsize $2$}}

\put(0,35){\circle*{2}} \put(20,35){\circle*{2}}

\put(0,40){\makebox(0,0)[b]{\scriptsize $0$}}
\put(20,40){\makebox(0,0)[b]{\scriptsize $1$}}

\put(0,35){\line(2,-1){40}} \put(20,35){\line(0,-1){20}}

\put(60,15){\circle*{2}} \put(80,15){\circle*{2}}

\put(60,10){\makebox(0,0)[t]{\scriptsize $3$}}
\put(80,10){\makebox(0,0)[t]{\scriptsize $4$}}

\put(40,35){\circle*{2}} \put(60,35){\circle*{2}}
\put(80,35){\circle*{2}}

\put(40,40){\makebox(0,0)[b]{\scriptsize $2$}}
\put(60,40){\makebox(0,0)[b]{\scriptsize $3$}}
\put(80,40){\makebox(0,0)[b]{\scriptsize $4$}}

\put(40,35){\line(1,-1){20}} \put(60,35){\line(0,-1){20}}
\put(80,35){\line(0,-1){20}}

\end{picture}
\end{center}
It is clear that $(\seto,+,0)$ is a strict monoidal category (see
\cite[Section~VII.1]{ML71}). This category is skeletal in the
sense of \cite[Section~IV.4]{ML71}, i.e.\ any two objects
isomorphic in it are identical. So this category is a skeleton of
itself, where a \emph{skeleton} of a category $\mathcal{K}$ is any
full subcategory $\mathcal{A}$ of $\mathcal{K}$ such that each
object of $\mathcal{K}$ is isomorphic in $\mathcal{K}$ to exactly
one object of $\mathcal{A}$. (The category $(\seto,+,0)$ is a PROP
in the sense of \cite[Chapter~V]{ML65}.)

The unique arrow $\check{\kappa}\!:0\rightarrow n$ is the empty function.
The first injection $\check{k}^1\!:n\rightarrow n+m$ and the second
injection $\check{k}^2\!:m\rightarrow n+m$ are as expected:

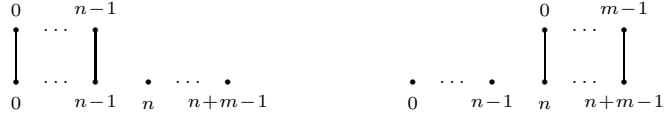
\begin{center}
\begin{picture}(230,50)
\put(0,15){\circle*{2}} \put(30,15){\circle*{2}}
\put(50,15){\circle*{2}} \put(80,15){\circle*{2}}

\put(0,5){\makebox(0,0)[b]{\scriptsize $0$}}
\put(30,5){\makebox(0,0)[b]{\scriptsize $n\mn 1$}}
\put(50,5){\makebox(0,0)[b]{\scriptsize $n$}}
\put(80,5){\makebox(0,0)[b]{\scriptsize $n\pl m\mn 1$}}
\put(15,15){\makebox(0,0){\scriptsize $\ldots$}}
\put(65,15){\makebox(0,0){\scriptsize $\ldots$}}

\put(0,35){\circle*{2}} \put(30,35){\circle*{2}}

\put(0,40){\makebox(0,0)[b]{\scriptsize $0$}}
\put(30,40){\makebox(0,0)[b]{\scriptsize $n\mn 1$}}
\put(15,35){\makebox(0,0){\scriptsize $\ldots$}}

\put(0,35){\line(0,-1){20}} \put(30,35){\line(0,-1){20}}

\put(150,15){\circle*{2}} \put(180,15){\circle*{2}}
\put(200,15){\circle*{2}} \put(230,15){\circle*{2}}

\put(150,5){\makebox(0,0)[b]{\scriptsize $0$}}
\put(180,5){\makebox(0,0)[b]{\scriptsize $n\mn 1$}}
\put(200,5){\makebox(0,0)[b]{\scriptsize $n$}}
\put(230,5){\makebox(0,0)[b]{\scriptsize $n\pl m\mn 1$}}
\put(165,15){\makebox(0,0){\scriptsize $\ldots$}}
\put(215,15){\makebox(0,0){\scriptsize $\ldots$}}

\put(200,35){\circle*{2}} \put(230,35){\circle*{2}}

\put(200,40){\makebox(0,0)[b]{\scriptsize $0$}}
\put(230,40){\makebox(0,0)[b]{\scriptsize $m\mn 1$}}
\put(215,35){\makebox(0,0){\scriptsize $\ldots$}}

\put(200,35){\line(0,-1){20}} \put(230,35){\line(0,-1){20}}

\end{picture}
\end{center}
The codiagonal arrow $\check{w}\!:n+n\rightarrow n$ is given by
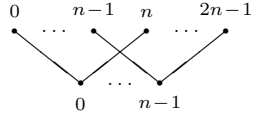
\begin{center}
\begin{picture}(80,50)
\put(25,15){\circle*{2}} \put(55,15){\circle*{2}}

\put(25,5){\makebox(0,0)[b]{\scriptsize $0$}}
\put(55,5){\makebox(0,0)[b]{\scriptsize $n\mn 1$}}
\put(40,15){\makebox(0,0){\scriptsize $\ldots$}}

\put(0,35){\circle*{2}} \put(30,35){\circle*{2}}
\put(50,35){\circle*{2}} \put(80,35){\circle*{2}}

\put(0,40){\makebox(0,0)[b]{\scriptsize $0$}}
\put(30,40){\makebox(0,0)[b]{\scriptsize $n\mn 1$}}
\put(50,40){\makebox(0,0)[b]{\scriptsize $n$}}
\put(80,40){\makebox(0,0)[b]{\scriptsize $2n\mn 1$}}
\put(15,35){\makebox(0,0){\scriptsize $\ldots$}}
\put(65,35){\makebox(0,0){\scriptsize $\ldots$}}

\put(0,35){\line(5,-4){25}} \put(30,35){\line(5,-4){25}}
\put(50,35){\line(-5,-4){25}} \put(80,35){\line(-5,-4){25}}
\end{picture}
\end{center}
Note that, for complete precision, we would have to label
$\check{k}^1$ and $\check{k}^2$ above by $n$ and $m$; analogously
we would have to label $\check{\kappa}$ and $\check{w}$ by $n$. We
will however omit these labels, since they can be reconstructed
from the contexts.

For arrows $f\!:n\rightarrow p$ and $g\!:m\rightarrow p$, the unique arrow
$[f,g]\!:n+m\rightarrow p$, which makes the following diagram commute
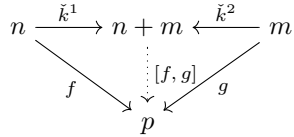
\[
\begin{tikzcd}
n \arrow{rd}[swap]{f} \arrow{r}{\check{k}^1} &n+m
\arrow[dotted]{d}{[f,g]}
\arrow[leftarrow]{r}{\check{k}^2} &m \arrow{ld}{g} \\
\phantom{.}&p &\phantom{.}
\end{tikzcd}
\]
is the composition $\check{w}\cirk(f+g)$.

Let $\catp$ be the subcategory of $\cat$ whose objects are the
categories with finite, strict monoidal, coproducts and whose
arrows are functors preserving this structure ``on the nose''. We
have just shown that $\seto$ is an object of $\catp$. Let
$\mathcal{D}$ (where D comes for disjunction) be a category with
finite, strict monoidal, coproducts freely generated by the set
$\{p\}$ of objects. This category is the image of $\{p\}$ under
the left adjoint of the forgetful functor from $\catp$ to $\set$
that maps a category to the set of its objects. One can construct
the category $\mathcal{D}$ out of syntactic material, but we will
not go into this construction here (see~\cite[Chapter~2]{DP04}).

Since $\mathcal{D}$ is freely generated by $\{p\}$, there is a
unique arrow $F\!:\mathcal{D}\rightarrow \seto$ of $\catp$, which extends
the function from $\{p\}$ to $\omega$ mapping $p$ to 1. The
following proposition stems from \cite{K72a} (see p.\ 129, where a
related dual result is announced), \cite{M80} (Theorem 2.2),
\cite{TS96} (Theorem 8.2.3, p.\ 207), \cite{P02} (Section~7)
and~\cite{DP01}.

\begin{prop}
The functor $F$ is an isomorphism.
\end{prop}

\begin{proof} It is trivial to check that $F$ is one-one and onto on objects. The faithfulness of
$F$ is a result just dual to Cartesian Coherence of
\cite[\S9.2]{DP04}. It remains to show that $F$ is full, and since
$F$ is one-one on objects, it is enough to verify that an
arbitrary function $f\!:n\rightarrow m$ can be built in terms of
identities, composition, empty functions, injections and brackets
[,]. We proceed by induction on $n\geq 0$.

If $n=0$, then $f=\check{\kappa}$, the empty function. If $n=m=1$,
then $f$ is the identity. Suppose $n=1$ and $m>1$. If $f(0)=0$,
then $f=\check{k}^1$, and if $f(0)=m\mn 1$, then $f=\check{k}^2$.
If $f(0)=i$ for $0<i<m\mn 1$, then $f=\check{k}^2\cirk
\check{k}^1$, for $\check{k}^1\!:1\rightarrow m\mn i$ and
$\check{k}^2\!:m\mn i\rightarrow m$.

If $n=i+1$, for $i\geq 1$, then $f=[f_1,f_2]$ for $f_1\!:i\rightarrow m$
and $f_2\!:1\rightarrow m$, and we just apply the induction hypothesis to
$f_1$ and $f_2$.
\end{proof}

\section{Products in $\seto$}
Besides the structure given by finite coproducts, the category
$\seto$ has the structure given by finite products. The empty
product in $\seto$ is the object 1 and the binary product on
objects is given by multiplication. As in the case of coproducts,
we will omit the labels of all the arrows since they can be
reconstructed from the contexts.

Let $\iota\!:n\times m\rightarrow n\cdot m$ be the bijection defined by
$\iota(i,j)=i\cdot m+j$. The inverse of this function is defined
by $\iota^{-1}(i)=(\lfloor i/m\rfloor,i\,\mbox{\rm mod}\,m)$,
where $\lfloor i/m\rfloor$ is the quotient and $i\,\mbox{\rm
mod}\,m$ is the remainder for the division of $i$ by $m$. In
other words, $\iota^{-1}(i)$ is the element at the $i$-th place
(starting from 0) in the lexicographically ordered set $n\times
m$. For example, if $n=2$ and $m=3$, then $\iota^{-1}(1)=(0,1)$
and $\iota^{-1}(5)=(1,2)$.

The product $f_1\cdot f_2\!:n_1\cdot n_2\rightarrow m_1\cdot m_2$ of
arrows $f_1\!:n_1\rightarrow m_1$ and $f_2\!:n_2\rightarrow m_2$ is defined as
$\iota\cirk (f_1\times f_2)\cirk \iota^{-1}$, where $\times$ is
the standard product in $\set$, i.e.,
\[
(f_1\times f_2)(i_1,i_2)=(f_1(i_1),f_2(i_2)).
\]
For example, if $f_1\!:2\rightarrow 2$ and $f_2\!:3\rightarrow 2$ are given, respectively,
by
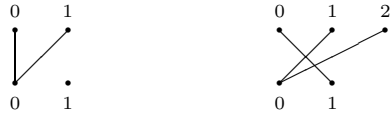
\begin{center}
\begin{picture}(140,50)
\put(0,15){\circle*{2}} \put(20,15){\circle*{2}}

\put(0,5){\makebox(0,0)[b]{\scriptsize $0$}}
\put(20,5){\makebox(0,0)[b]{\scriptsize $1$}}

\put(0,35){\circle*{2}} \put(20,35){\circle*{2}}

\put(0,40){\makebox(0,0)[b]{\scriptsize $0$}}
\put(20,40){\makebox(0,0)[b]{\scriptsize $1$}}

\put(0,35){\line(0,-1){20}} \put(20,35){\line(-1,-1){20}}

\put(100,15){\circle*{2}} \put(120,15){\circle*{2}}

\put(100,5){\makebox(0,0)[b]{\scriptsize $0$}}
\put(120,5){\makebox(0,0)[b]{\scriptsize $1$}}

\put(100,35){\circle*{2}} \put(120,35){\circle*{2}}
\put(140,35){\circle*{2}}

\put(100,40){\makebox(0,0)[b]{\scriptsize $0$}}
\put(120,40){\makebox(0,0)[b]{\scriptsize $1$}}
\put(140,40){\makebox(0,0)[b]{\scriptsize $2$}}

\put(100,35){\line(1,-1){20}} \put(120,35){\line(-1,-1){20}}
\put(140,35){\line(-2,-1){40}}

\end{picture}
\end{center}
then $f_1\cdot f_2\!:6\rightarrow 4$ and $f_2\cdot f_1\!:6\rightarrow 4$ are
given, respectively, by
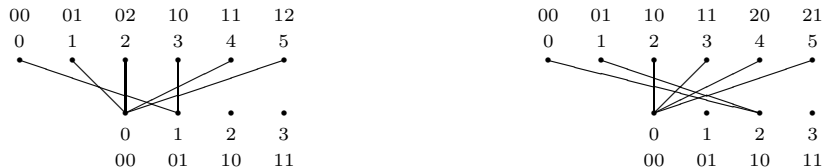
\begin{center}
\begin{picture}(300,70)
\put(40,25){\circle*{2}} \put(60,25){\circle*{2}}
\put(80,25){\circle*{2}} \put(100,25){\circle*{2}}

\put(40,15){\makebox(0,0)[b]{\scriptsize $0$}}
\put(60,15){\makebox(0,0)[b]{\scriptsize $1$}}
\put(80,15){\makebox(0,0)[b]{\scriptsize $2$}}
\put(100,15){\makebox(0,0)[b]{\scriptsize $3$}}

\put(40,5){\makebox(0,0)[b]{\scriptsize $00$}}
\put(60,5){\makebox(0,0)[b]{\scriptsize $01$}}
\put(80,5){\makebox(0,0)[b]{\scriptsize $10$}}
\put(100,5){\makebox(0,0)[b]{\scriptsize $11$}}

\put(0,45){\circle*{2}} \put(20,45){\circle*{2}}
\put(40,45){\circle*{2}} \put(60,45){\circle*{2}}
\put(80,45){\circle*{2}} \put(100,45){\circle*{2}}

\put(0,50){\makebox(0,0)[b]{\scriptsize $0$}}
\put(20,50){\makebox(0,0)[b]{\scriptsize $1$}}
\put(40,50){\makebox(0,0)[b]{\scriptsize $2$}}
\put(60,50){\makebox(0,0)[b]{\scriptsize $3$}}
\put(80,50){\makebox(0,0)[b]{\scriptsize $4$}}
\put(100,50){\makebox(0,0)[b]{\scriptsize $5$}}

\put(0,60){\makebox(0,0)[b]{\scriptsize $00$}}
\put(20,60){\makebox(0,0)[b]{\scriptsize $01$}}
\put(40,60){\makebox(0,0)[b]{\scriptsize $02$}}
\put(60,60){\makebox(0,0)[b]{\scriptsize $10$}}
\put(80,60){\makebox(0,0)[b]{\scriptsize $11$}}
\put(100,60){\makebox(0,0)[b]{\scriptsize $12$}}

\put(0,45){\line(3,-1){60}} \put(20,45){\line(1,-1){20}}
\put(40,45){\line(0,-1){20}} \put(60,45){\line(0,-1){20}}
\put(80,45){\line(-2,-1){40}} \put(100,45){\line(-3,-1){60}}

\put(240,25){\circle*{2}} \put(260,25){\circle*{2}}
\put(280,25){\circle*{2}} \put(300,25){\circle*{2}}

\put(240,15){\makebox(0,0)[b]{\scriptsize $0$}}
\put(260,15){\makebox(0,0)[b]{\scriptsize $1$}}
\put(280,15){\makebox(0,0)[b]{\scriptsize $2$}}
\put(300,15){\makebox(0,0)[b]{\scriptsize $3$}}

\put(240,5){\makebox(0,0)[b]{\scriptsize $00$}}
\put(260,5){\makebox(0,0)[b]{\scriptsize $01$}}
\put(280,5){\makebox(0,0)[b]{\scriptsize $10$}}
\put(300,5){\makebox(0,0)[b]{\scriptsize $11$}}

\put(200,45){\circle*{2}} \put(220,45){\circle*{2}}
\put(240,45){\circle*{2}} \put(260,45){\circle*{2}}
\put(280,45){\circle*{2}} \put(300,45){\circle*{2}}

\put(200,50){\makebox(0,0)[b]{\scriptsize $0$}}
\put(220,50){\makebox(0,0)[b]{\scriptsize $1$}}
\put(240,50){\makebox(0,0)[b]{\scriptsize $2$}}
\put(260,50){\makebox(0,0)[b]{\scriptsize $3$}}
\put(280,50){\makebox(0,0)[b]{\scriptsize $4$}}
\put(300,50){\makebox(0,0)[b]{\scriptsize $5$}}

\put(200,60){\makebox(0,0)[b]{\scriptsize $00$}}
\put(220,60){\makebox(0,0)[b]{\scriptsize $01$}}
\put(240,60){\makebox(0,0)[b]{\scriptsize $10$}}
\put(260,60){\makebox(0,0)[b]{\scriptsize $11$}}
\put(280,60){\makebox(0,0)[b]{\scriptsize $20$}}
\put(300,60){\makebox(0,0)[b]{\scriptsize $21$}}

\put(200,45){\line(4,-1){80}} \put(220,45){\line(3,-1){60}}
\put(240,45){\line(0,-1){20}} \put(260,45){\line(-1,-1){20}}
\put(280,45){\line(-2,-1){40}} \put(300,45){\line(-3,-1){60}}

\end{picture}
\end{center}
It is easy to verify that the structure on $\seto$ given by
$\cdot$ and 1 is again strict monoidal. Moreover, since $\seto$ is
skeletal, the product commutes on objects; we always have $n\cdot
m=m\cdot n$ but this does not mean, as it is shown above, that we
always have $f_1\cdot f_2=f_2\cdot f_1$.

For every $n\in \omega$ there is a unique function
$\hat{\kappa}\!:n\rightarrow 1$. The first projection $\hat{k}^1\!:n\cdot
m\rightarrow n$ is defined as $\pi_1\cirk \iota^{-1}$, where
$\pi_1\!:n\times m\rightarrow n$ is the ordinary first projection, and
analogously for $\hat{k}^2\!:n\cdot m\rightarrow m$. For example, the
first projection $\hat{k}^1\!:6\rightarrow 2$ is given by
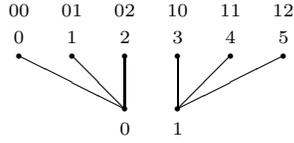
\begin{center}
\begin{picture}(100,60)(0,10)
\put(40,25){\circle*{2}} \put(60,25){\circle*{2}}

\put(40,15){\makebox(0,0)[b]{\scriptsize $0$}}
\put(60,15){\makebox(0,0)[b]{\scriptsize $1$}}

\put(0,45){\circle*{2}} \put(20,45){\circle*{2}}
\put(40,45){\circle*{2}} \put(60,45){\circle*{2}}
\put(80,45){\circle*{2}} \put(100,45){\circle*{2}}

\put(0,50){\makebox(0,0)[b]{\scriptsize $0$}}
\put(20,50){\makebox(0,0)[b]{\scriptsize $1$}}
\put(40,50){\makebox(0,0)[b]{\scriptsize $2$}}
\put(60,50){\makebox(0,0)[b]{\scriptsize $3$}}
\put(80,50){\makebox(0,0)[b]{\scriptsize $4$}}
\put(100,50){\makebox(0,0)[b]{\scriptsize $5$}}

\put(0,60){\makebox(0,0)[b]{\scriptsize $00$}}
\put(20,60){\makebox(0,0)[b]{\scriptsize $01$}}
\put(40,60){\makebox(0,0)[b]{\scriptsize $02$}}
\put(60,60){\makebox(0,0)[b]{\scriptsize $10$}}
\put(80,60){\makebox(0,0)[b]{\scriptsize $11$}}
\put(100,60){\makebox(0,0)[b]{\scriptsize $12$}}

\put(0,45){\line(2,-1){40}} \put(20,45){\line(1,-1){20}}
\put(40,45){\line(0,-1){20}} \put(60,45){\line(0,-1){20}}
\put(80,45){\line(-1,-1){20}} \put(100,45){\line(-2,-1){40}}

\end{picture}
\end{center}
while the second projection $\hat{k}^2\!:6\rightarrow 3$ is given by
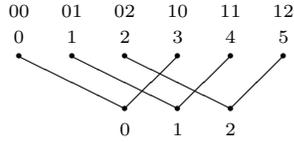
\begin{center}
\begin{picture}(100,60)(0,10)
\put(40,25){\circle*{2}} \put(60,25){\circle*{2}}
\put(80,25){\circle*{2}}

\put(40,15){\makebox(0,0)[b]{\scriptsize $0$}}
\put(60,15){\makebox(0,0)[b]{\scriptsize $1$}}
\put(80,15){\makebox(0,0)[b]{\scriptsize $2$}}

\put(0,45){\circle*{2}} \put(20,45){\circle*{2}}
\put(40,45){\circle*{2}} \put(60,45){\circle*{2}}
\put(80,45){\circle*{2}} \put(100,45){\circle*{2}}

\put(0,50){\makebox(0,0)[b]{\scriptsize $0$}}
\put(20,50){\makebox(0,0)[b]{\scriptsize $1$}}
\put(40,50){\makebox(0,0)[b]{\scriptsize $2$}}
\put(60,50){\makebox(0,0)[b]{\scriptsize $3$}}
\put(80,50){\makebox(0,0)[b]{\scriptsize $4$}}
\put(100,50){\makebox(0,0)[b]{\scriptsize $5$}}

\put(0,60){\makebox(0,0)[b]{\scriptsize $00$}}
\put(20,60){\makebox(0,0)[b]{\scriptsize $01$}}
\put(40,60){\makebox(0,0)[b]{\scriptsize $02$}}
\put(60,60){\makebox(0,0)[b]{\scriptsize $10$}}
\put(80,60){\makebox(0,0)[b]{\scriptsize $11$}}
\put(100,60){\makebox(0,0)[b]{\scriptsize $12$}}

\put(0,45){\line(2,-1){40}} \put(20,45){\line(2,-1){40}}
\put(40,45){\line(2,-1){40}} \put(60,45){\line(-1,-1){20}}
\put(80,45){\line(-1,-1){20}} \put(100,45){\line(-1,-1){20}}

\end{picture}
\end{center}
The diagonal arrow $\hat{w}\!:n\rightarrow n\cdot n$ is defined as
$\iota\cirk\Delta$, where $\Delta\!:n\rightarrow n\times n$ is the
ordinary diagonal map. For example, $\hat{w}\!:3\rightarrow 9$ is
given by
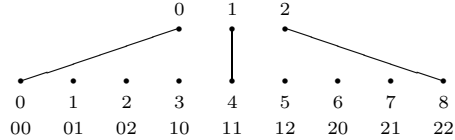
\begin{center}
\begin{picture}(160,60)
\put(0,25){\circle*{2}} \put(20,25){\circle*{2}}
\put(40,25){\circle*{2}} \put(60,25){\circle*{2}}
\put(80,25){\circle*{2}} \put(100,25){\circle*{2}}
\put(120,25){\circle*{2}} \put(140,25){\circle*{2}}
\put(160,25){\circle*{2}}

\put(0,15){\makebox(0,0)[b]{\scriptsize $0$}}
\put(20,15){\makebox(0,0)[b]{\scriptsize $1$}}
\put(40,15){\makebox(0,0)[b]{\scriptsize $2$}}
\put(60,15){\makebox(0,0)[b]{\scriptsize $3$}}
\put(80,15){\makebox(0,0)[b]{\scriptsize $4$}}
\put(100,15){\makebox(0,0)[b]{\scriptsize $5$}}
\put(120,15){\makebox(0,0)[b]{\scriptsize $6$}}
\put(140,15){\makebox(0,0)[b]{\scriptsize $7$}}
\put(160,15){\makebox(0,0)[b]{\scriptsize $8$}}

\put(0,5){\makebox(0,0)[b]{\scriptsize $00$}}
\put(20,5){\makebox(0,0)[b]{\scriptsize $01$}}
\put(40,5){\makebox(0,0)[b]{\scriptsize $02$}}
\put(60,5){\makebox(0,0)[b]{\scriptsize $10$}}
\put(80,5){\makebox(0,0)[b]{\scriptsize $11$}}
\put(100,5){\makebox(0,0)[b]{\scriptsize $12$}}
\put(120,5){\makebox(0,0)[b]{\scriptsize $20$}}
\put(140,5){\makebox(0,0)[b]{\scriptsize $21$}}
\put(160,5){\makebox(0,0)[b]{\scriptsize $22$}}

\put(60,45){\circle*{2}} \put(80,45){\circle*{2}}
\put(100,45){\circle*{2}}

\put(60,50){\makebox(0,0)[b]{\scriptsize $0$}}
\put(80,50){\makebox(0,0)[b]{\scriptsize $1$}}
\put(100,50){\makebox(0,0)[b]{\scriptsize $2$}}

\put(60,45){\line(-3,-1){60}} \put(80,45){\line(0,-1){20}}
\put(100,45){\line(3,-1){60}}

\end{picture}
\end{center}

For arrows $f\!:p\rightarrow n$ and $g\!:p\rightarrow m$, the unique arrow
$\langle f,g\rangle\!:p\rightarrow n\cdot m$, which makes the following
diagram commute
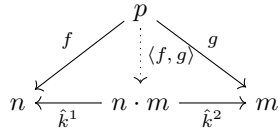
\[
\begin{tikzcd}
\phantom{.}& \arrow{ld}[swap]{f} p \arrow[dotted]{d}{\langle
f,g\rangle} \arrow{rd}{g}&\phantom{.}
\\
n  \arrow[leftarrow]{r}[swap]{\hat{k}^1} &n\cdot m
\arrow[rightarrow]{r}[swap]{\hat{k}^2} &m
\end{tikzcd}
\]
is the composition $(f\cdot g)\cirk\hat{w}$.

\section{Mapping $\mathcal{C}$ into $\seto$}
Let $\catt$ be the subcategory of $\cat$ whose objects are the
categories with finite, strict monoidal, products and whose arrows
are functors preserving this structure ``on the nose''. We have
just shown that $\seto$ is an object of $\catt$. Let $\mathcal{C}$
(where C stands for conjunction) be a category with finite strict
monoidal products freely generated by a countable set
$P=\{p_1,p_2,\ldots\}$ of objects. This category is the image of
$P$ under the left adjoint for the forgetful functor from $\catt$
to $\set$ that maps a category to the set of its objects. The
category $\mathcal{C}$ can as $\mathcal{D}$ be constructed out of
syntactic material, but we will not go into this construction here
(see~\cite[Chapter~2]{DP04}). We assume only that the objects of
$\mathcal{C}$ are the finite sequences of elements of $P$.

In the proof below we rely on prime numbers, because when natural numbers greater than 0 are generated with multiplication and 1, the prime numbers are free generators. Disjunctive logic with a single propositional variable corresponds to generating the natural numbers with addition and 0 out of the free generator 1, while conjunctive logic with countably many propositional variables corresponds to building them with multiplication and 1 with the prime numbers as free generators. Using the prime numbers enables us to obtain that the functor mentioned in Proposition \ref{iso} is one-one on objects.

Since $\mathcal{C}$ is freely generated by $P$, there is a unique
arrow ${H\!:\mathcal{C}\rightarrow \seto}$ of $\catt$, which extends the
function from $P$ to $\omega$ mapping $p_n$ to the $n$-th prime
number $\mathbf{p}_n$. Our goal is to prove the following
proposition.

\begin{prop}\label{tvrdjenje1}
The functor $H$ is faithful.
\end{prop}

For this, we rely on Cartesian Coherence of \cite[\S9.2]{DP04} and
Proposition~5 of \cite[\S5]{DP03a}. We start with some auxiliary
notions.

Since $\seto$ is an object of $\catp$, its opposite category
$\seto^{op}$, with $+$ as the product and $0$ as the terminal
object, is an object of $\catt$. Hence, there is a unique arrow
${G\!:\mathcal{C}\rightarrow \seto^{op}}$ of $\catt$, which extends the
function from $P$ to $\omega$ mapping every $p_n$ to $1$.

The following lemma follows from Cartesian Coherence of \cite[\S9.2]{DP04}.

\begin{lem}\label{lema1}
The functor $G$ is faithful.
\end{lem}

Let \emph{Gen} be a category whose objects are again the finite
ordinals. An arrow of \emph{Gen} from $n$ to $m$ is an equivalence
relation defined on $n + m$, called a \emph{split
equivalence}. (This is any equivalence relation on the ordinal $n + m$, and it is called \emph{split}, because its domain $n + m$ is divided into the source $n$ and the target $m$; for more on these matters see \cite{DP09}.) The identity arrow from $n$ to $n$ is the split
equivalence with $n$ equivalence classes of the form
$\{i,i+n\}$. (We follow here the presentation of \emph{Gen} in \cite{DP03a},
rather than that in \cite{DP09}, which yields an isomorphic category.)

Composition of arrows is defined, roughly speaking, as the
transitive closure of the union of the two relations composed,
where we omit the ordered pairs one of whose members is in the
middle (see \cite[Section~2]{DP03a}, \cite[Section~2]{DP03b} and
\cite[Section~2]{DP09} for a detailed definition). For example,
the split equivalences $R_1$ and $R_2$ given, respectively, by
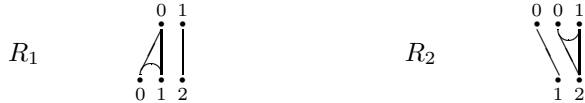
\begin{center}
\begin{picture}(264,40)

\put(50,10){\circle*{2}} \put(58,10){\circle*{2}}
\put(66,10){\circle*{2}} \put(58,31){\circle*{2}}
\put(66,31){\circle*{2}}

\put(58,29){\line(-1,-2){7.6}} \put(58,12){\line(0,1){17}}
\put(66,12){\line(0,1){17}}

\put(54,12){\oval(8,8)[t]}

\put(50,7){\makebox(0,0)[t]{\scriptsize$0$}}
\put(58,7){\makebox(0,0)[t]{\scriptsize$1$}}
\put(66,7){\makebox(0,0)[t]{\scriptsize$2$}}
\put(58,34){\makebox(0,0)[b]{\scriptsize$0$}}
\put(66,34){\makebox(0,0)[b]{\scriptsize$1$}}

\put(0,20){\makebox(0,0)[l]{$R_1$}}

\put(200,31){\circle*{2}} \put(208,10){\circle*{2}}
\put(216,10){\circle*{2}} \put(208,31){\circle*{2}}
\put(216,31){\circle*{2}}

\put(216,12){\line(-1,2){7.6}} \put(216,12){\line(0,1){17}}
\put(208.5,12){\line(-1,2){8.5}}

\put(212,29){\oval(8,8)[b]}

\put(200,34){\makebox(0,0)[b]{\scriptsize$0$}}
\put(208,7){\makebox(0,0)[t]{\scriptsize$1$}}
\put(216,7){\makebox(0,0)[t]{\scriptsize$2$}}
\put(208,34){\makebox(0,0)[b]{\scriptsize$0$}}
\put(216,34){\makebox(0,0)[b]{\scriptsize$1$}}

\put(150,20){\makebox(0,0)[l]{$R_2$}}

\end{picture}
\end{center}
are composed so as to yield the split equivalence
${R_2\cirk R_1}$ given by the picture on the right-hand side of the equation sign:
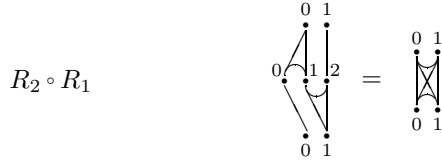
\begin{center}
\begin{picture}(164,60)

\put(50,30){\circle*{2}} \put(58,30){\circle*{2}}
\put(66,30){\circle*{2}} \put(58,9){\circle*{2}}
\put(66,9){\circle*{2}} \put(58,51){\circle*{2}}
\put(66,51){\circle*{2}}

\put(58,49){\line(-1,-2){7.6}} \put(58,32){\line(0,1){17}}
\put(66,32){\line(0,1){17}}

\put(54,32){\oval(8,8)[t]}

\put(66,11){\line(-1,2){7.6}} \put(66,11){\line(0,1){17}}
\put(58.5,11){\line(-1,2){8.5}}

\put(62,28){\oval(8,8)[b]}

\put(58,6){\makebox(0,0)[t]{\scriptsize$0$}}
\put(66,6){\makebox(0,0)[t]{\scriptsize$1$}}
\put(58,55){\makebox(0,0)[b]{\scriptsize$0$}}
\put(66,55){\makebox(0,0)[b]{\scriptsize$1$}}
\put(49,32){\makebox(0,0)[br]{\scriptsize$0$}}
\put(59,32){\makebox(0,0)[bl]{\scriptsize$1$}}
\put(67,32){\makebox(0,0)[bl]{\scriptsize$2$}}

\put(-54,30){\makebox(0,0)[l]{$R_2\cirk R_1$}}

\put(100,19){\circle*{2}} \put(108,19){\circle*{2}}
\put(100,41){\circle*{2}} \put(108,41){\circle*{2}}
\put(104,21){\oval(8,8)[t]} \put(104,39){\oval(8,8)[b]}
\put(100.8,23.4){\line(1,2){7}} \put(107.2,23.4){\line(-1,2){7}}
\put(100,21){\line(0,1){18}} \put(108,21){\line(0,1){18}}
\put(100,16){\makebox(0,0)[t]{\scriptsize$0$}}
\put(108,16){\makebox(0,0)[t]{\scriptsize$1$}}
\put(100,44){\makebox(0,0)[b]{\scriptsize$0$}}
\put(108,44){\makebox(0,0)[b]{\scriptsize$1$}}

\put(83,30){\makebox(0,0){$=$}}

\end{picture}
\end{center}

Consider the function, which maps the arrow $f^{op}\!:n\rightarrow m$ of
$\seto^{op}$ to the split equivalence (an arrow of \emph{Gen})
between $n$ and $m$ with $n$ equivalence classes, one for each
$i\in n$, of the form
\[
\{i\}\cup\{j+n\mid j \in m \;{\rm and}\;f(j)=i\}.
\]
(For every value of the function $f$ we put in the same class together with this value all the arguments with that value; besides that we have singleton equivalence classes for elements of the codomain of $f$ that are not in the image of $f$.) For example, $f^{op}\!:3\rightarrow 4$ given by the picture on the
left-hand side is mapped to the split equivalence given by the
picture on the right-hand side:
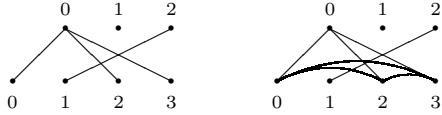
\begin{center}
\begin{picture}(160,50)

\put(0,15){\circle*{2}} \put(20,15){\circle*{2}}
\put(40,15){\circle*{2}} \put(60,15){\circle*{2}}

\put(20,35){\circle*{2}} \put(40,35){\circle*{2}}
\put(60,35){\circle*{2}}

\put(20,35){\line(-1,-1){20}} \put(20,35){\line(1,-1){20}}
\put(20,35){\line(2,-1){40}} \put(60,35){\line(-2,-1){40}}

\put(0,5){\makebox(0,0)[b]{\scriptsize$0$}}
\put(20,5){\makebox(0,0)[b]{\scriptsize$1$}}
\put(40,5){\makebox(0,0)[b]{\scriptsize$2$}}
\put(60,5){\makebox(0,0)[b]{\scriptsize$3$}}

\put(20,40){\makebox(0,0)[b]{\scriptsize$0$}}
\put(40,40){\makebox(0,0)[b]{\scriptsize$1$}}
\put(60,40){\makebox(0,0)[b]{\scriptsize$2$}}

\put(100,15){\circle*{2}} \put(120,15){\circle*{2}}
\put(140,15){\circle*{2}} \put(160,15){\circle*{2}}

\put(120,35){\circle*{2}} \put(140,35){\circle*{2}}
\put(160,35){\circle*{2}}

\put(120,35){\line(-1,-1){20}} \put(120,35){\line(1,-1){20}}
\put(120,35){\line(2,-1){40}} \put(160,35){\line(-2,-1){40}}

\put(100,5){\makebox(0,0)[b]{\scriptsize$0$}}
\put(120,5){\makebox(0,0)[b]{\scriptsize$1$}}
\put(140,5){\makebox(0,0)[b]{\scriptsize$2$}}
\put(160,5){\makebox(0,0)[b]{\scriptsize$3$}}

\put(120,40){\makebox(0,0)[b]{\scriptsize$0$}}
\put(140,40){\makebox(0,0)[b]{\scriptsize$1$}}
\put(160,40){\makebox(0,0)[b]{\scriptsize$2$}}

\qbezier(100,15)(120,25)(140,15) \qbezier(100,15)(130,30)(160,15)
\qbezier(140,15)(150,20)(160,15)

\end{picture}
\end{center}
It is not difficult to check that this is the arrow function of a
faithful functor $J\!:\seto^{op}\rightarrow \mbox{\emph{Gen}}$, which is
identity on objects (see \cite[end of Section~2]{DP09}).

Let $R\!:n\rightarrow m$ be an arrow of \emph{Gen}, and
$\overrightarrow{ab}$ the sequence $a_0\ldots a_{n-1} b_0 \ldots
b_{m-1}$ of not necessarily distinct finite ordinals greater
than or equal to 2. We define a relation (not a split equivalence)
$F_{\overrightarrow{ab}}(R)$ between the ordinal $a_0\cdot\ldots\cdot a_{n-1}$
(which is 1 when $n=0$) and the ordinal $b_0\cdot\ldots\cdot b_{m-1}$ in the
following way.

For $\overrightarrow{d}$ being the sequence $d_0\ldots d_{k-1}$, with $k\geq 0$, let
$\iota_{\overrightarrow{d}}\!:d_0\times\ldots\times d_{k-1}\rightarrow
d_0\cdot\ldots\cdot d_{k-1}$ be a function defined by
\[
\iota_{\overrightarrow{d}}(i_0,\ldots,i_{k-1})=i_0\cdot d_1
\cdot\ldots\cdot d_{k-1}+\ldots+i_{k-2}\cdot d_{n-1}+i_{k-1}.
\]
Its inverse $\iota_{\overrightarrow{d}}^{-1}\!:
d_0\cdot\ldots\cdot d_{k-1}\rightarrow d_0\times\ldots\times d_{k-1}$ is
defined by
\[
\iota_{\overrightarrow{d}}^{-1}(i)=(\lfloor i/(d_1\cdot\ldots\cdot
d_{k-1})\rfloor,\ldots, \lfloor i/ d_{k-1}\rfloor \,\mbox{\rm
mod}\, d_{k-2}, i\,\mbox{\rm mod}\,d_{k-1}).
\]
Note that $\iota$ and $\iota^{-1}$ defined at the beginning of
Section~3 are just $\iota_{\overrightarrow{d}}$ and
$\iota_{\overrightarrow{d}}^{-1}$ for $\overrightarrow{d}$ being
the two-element sequence $nm$. The bijection $\iota_{\overrightarrow{d}}$ assigns to an element of $d_0\times\ldots\times d_{k-1}$ its position in the lexicographical order.

Let $F_{\overrightarrow{ab}}(R)$ be defined as the set of all ordered pairs
$(i,j)$ in $(a_0\cdot\ldots\cdot a_{n-1})\times (b_0\cdot\ldots\cdot
b_{m-1})$ such that for $\pi_x$ the $x$-th projection, $\pi_y$ the $y$-th
projection and $\overrightarrow{c}$ the $n\pl m$-tuple obtained by
concatenating the $n$-tuple $\iota_{\overrightarrow{a}}^{-1}(i)$ with the
$m$-tuple $\iota_{\overrightarrow{b}}^{-1}(j)$ we have
\[
(\forall x,y\in n\pl m)\; ((x,y)\in R\Rightarrow
\pi_x(\overrightarrow{c})=\pi_y(\overrightarrow{c})).
\]

Roughly speaking, the idea is to connect by $F_{\overrightarrow{ab}}(R)$ an element of
$a_0\times\ldots\times a_{n-1}$ with an element of
$b_0\times\ldots\times b_{m-1}$ when this pair ``matches'' $R$.
For example, if $R\!:3\rightarrow 4$ is given by
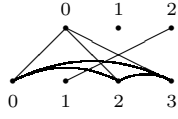
\begin{center}
\begin{picture}(60,50)

\put(0,15){\circle*{2}} \put(20,15){\circle*{2}}
\put(40,15){\circle*{2}} \put(60,15){\circle*{2}}

\put(20,35){\circle*{2}} \put(40,35){\circle*{2}}
\put(60,35){\circle*{2}}

\put(20,35){\line(-1,-1){20}} \put(20,35){\line(1,-1){20}}
\put(20,35){\line(2,-1){40}} \put(60,35){\line(-2,-1){40}}

\put(0,5){\makebox(0,0)[b]{\scriptsize$0$}}
\put(20,5){\makebox(0,0)[b]{\scriptsize$1$}}
\put(40,5){\makebox(0,0)[b]{\scriptsize$2$}}
\put(60,5){\makebox(0,0)[b]{\scriptsize$3$}}

\put(20,40){\makebox(0,0)[b]{\scriptsize$0$}}
\put(40,40){\makebox(0,0)[b]{\scriptsize$1$}}
\put(60,40){\makebox(0,0)[b]{\scriptsize$2$}}

\qbezier(0,15)(20,25)(40,15) \qbezier(0,15)(30,30)(60,15)
\qbezier(40,15)(50,20)(60,15)

\end{picture}
\end{center}
and $a_0=3$, $a_1=a_2=b_0=\ldots=b_3=2$, then $2$ from $3\cdot
2^2$, which corresponds to the triple $(0,1,0)$, is connected to $0$ from $2^4$, which corresponds to the
quadruple $(0,0,0,0)$, and the pair $((0,1,0),(0,0,0,0))$ matches $R$, since we have the picture
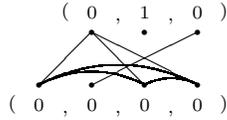
\begin{center}
\begin{picture}(60,50)

\put(0,15){\circle*{2}} \put(20,15){\circle*{2}}
\put(40,15){\circle*{2}} \put(60,15){\circle*{2}}

\put(20,35){\circle*{2}} \put(40,35){\circle*{2}}
\put(60,35){\circle*{2}}

\put(20,35){\line(-1,-1){20}} \put(20,35){\line(1,-1){20}}
\put(20,35){\line(2,-1){40}} \put(60,35){\line(-2,-1){40}}

\put(-10,4){\makebox(0,0)[b]{\scriptsize$($}}
\put(0,5){\makebox(0,0)[b]{\scriptsize$0$}}
\put(10,3){\makebox(0,0)[b]{\scriptsize$,$}}
\put(20,5){\makebox(0,0)[b]{\scriptsize$0$}}
\put(30,3){\makebox(0,0)[b]{\scriptsize$,$}}
\put(40,5){\makebox(0,0)[b]{\scriptsize$0$}}
\put(50,3){\makebox(0,0)[b]{\scriptsize$,$}}
\put(60,5){\makebox(0,0)[b]{\scriptsize$0$}}
\put(70,4){\makebox(0,0)[b]{\scriptsize$)$}}

\put(10,39){\makebox(0,0)[b]{\scriptsize$($}}
\put(20,40){\makebox(0,0)[b]{\scriptsize$0$}}
\put(30,38){\makebox(0,0)[b]{\scriptsize$,$}}
\put(40,40){\makebox(0,0)[b]{\scriptsize$1$}}
\put(50,38){\makebox(0,0)[b]{\scriptsize$,$}}
\put(60,40){\makebox(0,0)[b]{\scriptsize$0$}}
\put(70,39){\makebox(0,0)[b]{\scriptsize$)$}}

\qbezier(0,15)(20,25)(40,15) \qbezier(0,15)(30,30)(60,15)
\qbezier(40,15)(50,20)(60,15)

\end{picture}
\end{center}
which is as the picture given for $R$ with every element of an equivalence class of $R$ replaced by the same number. In the same manner, we conclude that $0,1,3,4,5,6$ and $7$ from
$3\cdot 2^2$ are connected, respectively, to
$0,4,4,11,15,11$ and $15$ from $2^4$, and that there are no other pairs
corresponding to the elements of $F_{\overrightarrow{ab}}(R)$. Hence, in this case $F_{\overrightarrow{ab}}(R)$ is not
a function.

For $R\!:n\rightarrow m$ an arrow of \emph{Gen} we say that a sequence of
finite ordinals $a_0\ldots a_{n-1} b_0 \ldots b_{m-1}$ is
\emph{appropriate} for $R$ when $(i,j)\in R$ and $i,j<n$ implies $a_i=a_j$, $(i,j)\in R$ and $i<n$ and $j\geq n$ implies $a_i=a_j$, and $(i,j)\in R$ and $i,j\geq n$ implies $b_i=b_j$. The following lemma, with a straightforward proof,
provides sufficient conditions for $F_{\overrightarrow{ab}}(R)$ to be a
function.

\begin{lem}\label{lema0}
For $R=J(f^{op})$ and $\overrightarrow{ab}$ a sequence appropriate
for $R$, we have that $F_{\overrightarrow{ab}}(R)$ is a function.
\end{lem}

\begin{exm}\label{pr1}
{\em Let $f^{op}\!:2+1\rightarrow 2$ be the first projection in
$\seto^{op}$ given by

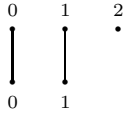
\begin{center}
\begin{picture}(80,40)
\put(0,10){\circle*{2}} \put(20,10){\circle*{2}}
\put(0,30){\circle*{2}} \put(20,30){\circle*{2}}
\put(40,30){\circle*{2}}

\put(0,10){\line(0,1){20}} \put(20,10){\line(0,1){20}}

\put(0,0){\makebox(0,0)[b]{\scriptsize$0$}}
\put(20,0){\makebox(0,0)[b]{\scriptsize$1$}}
\put(0,35){\makebox(0,0)[b]{\scriptsize$0$}}
\put(20,35){\makebox(0,0)[b]{\scriptsize$1$}}
\put(40,35){\makebox(0,0)[b]{\scriptsize$2$}}
\end{picture}
\end{center}

For $R=J(f^{op})$ and an appropriate sequence
$\overrightarrow{ab}=2\;3\;2\;\;2\;3$ we have that $F_{\overrightarrow{ab}}(R)$ is
given by

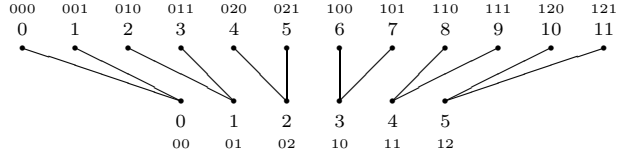
\begin{center}
\begin{picture}(220,60)
\put(60,20){\circle*{2}} \put(80,20){\circle*{2}}
\put(100,20){\circle*{2}} \put(120,20){\circle*{2}}
\put(140,20){\circle*{2}} \put(160,20){\circle*{2}}

\put(0,40){\circle*{2}} \put(20,40){\circle*{2}}
\put(40,40){\circle*{2}} \put(60,40){\circle*{2}}
\put(80,40){\circle*{2}} \put(100,40){\circle*{2}}
\put(120,40){\circle*{2}} \put(140,40){\circle*{2}}
\put(160,40){\circle*{2}} \put(180,40){\circle*{2}}
\put(200,40){\circle*{2}} \put(220,40){\circle*{2}}

\put(60,20){\line(-3,1){60}} \put(60,20){\line(-2,1){40}}
\put(80,20){\line(-2,1){40}} \put(80,20){\line(-1,1){20}}
\put(100,20){\line(-1,1){20}} \put(100,20){\line(0,1){20}}

\put(160,20){\line(3,1){60}} \put(160,20){\line(2,1){40}}
\put(140,20){\line(2,1){40}} \put(140,20){\line(1,1){20}}
\put(120,20){\line(1,1){20}} \put(120,20){\line(0,1){20}}

\put(60,10){\makebox(0,0)[b]{\scriptsize$0$}}
\put(80,10){\makebox(0,0)[b]{\scriptsize$1$}}
\put(100,10){\makebox(0,0)[b]{\scriptsize$2$}}
\put(120,10){\makebox(0,0)[b]{\scriptsize$3$}}
\put(140,10){\makebox(0,0)[b]{\scriptsize$4$}}
\put(160,10){\makebox(0,0)[b]{\scriptsize$5$}}

\put(60,2){\makebox(0,0)[b]{\tiny$00$}}
\put(80,2){\makebox(0,0)[b]{\tiny$01$}}
\put(100,2){\makebox(0,0)[b]{\tiny$02$}}
\put(120,2){\makebox(0,0)[b]{\tiny$10$}}
\put(140,2){\makebox(0,0)[b]{\tiny$11$}}
\put(160,2){\makebox(0,0)[b]{\tiny$12$}}

\put(0,45){\makebox(0,0)[b]{\scriptsize$0$}}
\put(20,45){\makebox(0,0)[b]{\scriptsize$1$}}
\put(40,45){\makebox(0,0)[b]{\scriptsize$2$}}
\put(60,45){\makebox(0,0)[b]{\scriptsize$3$}}
\put(80,45){\makebox(0,0)[b]{\scriptsize$4$}}
\put(100,45){\makebox(0,0)[b]{\scriptsize$5$}}
\put(120,45){\makebox(0,0)[b]{\scriptsize$6$}}
\put(140,45){\makebox(0,0)[b]{\scriptsize$7$}}
\put(160,45){\makebox(0,0)[b]{\scriptsize$8$}}
\put(180,45){\makebox(0,0)[b]{\scriptsize$9$}}
\put(200,45){\makebox(0,0)[b]{\scriptsize$10$}}
\put(220,45){\makebox(0,0)[b]{\scriptsize$11$}}

\put(0,53){\makebox(0,0)[b]{\tiny$000$}}
\put(20,53){\makebox(0,0)[b]{\tiny$001$}}
\put(40,53){\makebox(0,0)[b]{\tiny$010$}}
\put(60,53){\makebox(0,0)[b]{\tiny$011$}}
\put(80,53){\makebox(0,0)[b]{\tiny$020$}}
\put(100,53){\makebox(0,0)[b]{\tiny$021$}}
\put(120,53){\makebox(0,0)[b]{\tiny$100$}}
\put(140,53){\makebox(0,0)[b]{\tiny$101$}}
\put(160,53){\makebox(0,0)[b]{\tiny$110$}}
\put(180,53){\makebox(0,0)[b]{\tiny$111$}}
\put(200,53){\makebox(0,0)[b]{\tiny$120$}}
\put(220,53){\makebox(0,0)[b]{\tiny$121$}}
\end{picture}
\end{center}
which for the isomorphism
\[
\iota^{-1}\!:(2\cdot3)\cdot 2\rightarrow (2\cdot3)\times 2
\]
is equal to the composition $\pi_1\cirk\iota^{-1}$ given by
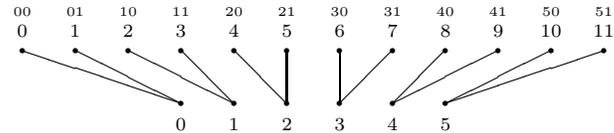
\begin{center}
\begin{picture}(220,60)
\put(60,20){\circle*{2}} \put(80,20){\circle*{2}}
\put(100,20){\circle*{2}} \put(120,20){\circle*{2}}
\put(140,20){\circle*{2}} \put(160,20){\circle*{2}}

\put(0,40){\circle*{2}} \put(20,40){\circle*{2}}
\put(40,40){\circle*{2}} \put(60,40){\circle*{2}}
\put(80,40){\circle*{2}} \put(100,40){\circle*{2}}
\put(120,40){\circle*{2}} \put(140,40){\circle*{2}}
\put(160,40){\circle*{2}} \put(180,40){\circle*{2}}
\put(200,40){\circle*{2}} \put(220,40){\circle*{2}}

\put(60,20){\line(-3,1){60}} \put(60,20){\line(-2,1){40}}
\put(80,20){\line(-2,1){40}} \put(80,20){\line(-1,1){20}}
\put(100,20){\line(-1,1){20}} \put(100,20){\line(0,1){20}}

\put(160,20){\line(3,1){60}} \put(160,20){\line(2,1){40}}
\put(140,20){\line(2,1){40}} \put(140,20){\line(1,1){20}}
\put(120,20){\line(1,1){20}} \put(120,20){\line(0,1){20}}

\put(60,10){\makebox(0,0)[b]{\scriptsize$0$}}
\put(80,10){\makebox(0,0)[b]{\scriptsize$1$}}
\put(100,10){\makebox(0,0)[b]{\scriptsize$2$}}
\put(120,10){\makebox(0,0)[b]{\scriptsize$3$}}
\put(140,10){\makebox(0,0)[b]{\scriptsize$4$}}
\put(160,10){\makebox(0,0)[b]{\scriptsize$5$}}

\put(0,45){\makebox(0,0)[b]{\scriptsize$0$}}
\put(20,45){\makebox(0,0)[b]{\scriptsize$1$}}
\put(40,45){\makebox(0,0)[b]{\scriptsize$2$}}
\put(60,45){\makebox(0,0)[b]{\scriptsize$3$}}
\put(80,45){\makebox(0,0)[b]{\scriptsize$4$}}
\put(100,45){\makebox(0,0)[b]{\scriptsize$5$}}
\put(120,45){\makebox(0,0)[b]{\scriptsize$6$}}
\put(140,45){\makebox(0,0)[b]{\scriptsize$7$}}
\put(160,45){\makebox(0,0)[b]{\scriptsize$8$}}
\put(180,45){\makebox(0,0)[b]{\scriptsize$9$}}
\put(200,45){\makebox(0,0)[b]{\scriptsize$10$}}
\put(220,45){\makebox(0,0)[b]{\scriptsize$11$}}

\put(0,53){\makebox(0,0)[b]{\tiny$00$}}
\put(20,53){\makebox(0,0)[b]{\tiny$01$}}
\put(40,53){\makebox(0,0)[b]{\tiny$10$}}
\put(60,53){\makebox(0,0)[b]{\tiny$11$}}
\put(80,53){\makebox(0,0)[b]{\tiny$20$}}
\put(100,53){\makebox(0,0)[b]{\tiny$21$}}
\put(120,53){\makebox(0,0)[b]{\tiny$30$}}
\put(140,53){\makebox(0,0)[b]{\tiny$31$}}
\put(160,53){\makebox(0,0)[b]{\tiny$40$}}
\put(180,53){\makebox(0,0)[b]{\tiny$41$}}
\put(200,53){\makebox(0,0)[b]{\tiny$50$}}
\put(220,53){\makebox(0,0)[b]{\tiny$51$}}
\end{picture}
\end{center}
Hence, this is the first projection from $(2\cdot 3)\cdot 2$ to
$2\cdot3$ in $\seto$.}
\end{exm}

\begin{exm}\label{pr2}
{\em Let $f^{op}\!:2\rightarrow 2+2$ be the diagonal arrow in
$\seto^{op}$ given by

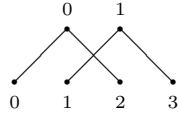
\begin{center}
\begin{picture}(80,40)
\put(0,10){\circle*{2}} \put(20,10){\circle*{2}}
\put(40,10){\circle*{2}} \put(60,10){\circle*{2}}
\put(20,30){\circle*{2}} \put(40,30){\circle*{2}}

\put(20,30){\line(-1,-1){20}} \put(40,30){\line(-1,-1){20}}
\put(20,30){\line(1,-1){20}} \put(40,30){\line(1,-1){20}}

\put(0,0){\makebox(0,0)[b]{\scriptsize$0$}}
\put(20,0){\makebox(0,0)[b]{\scriptsize$1$}}
\put(40,0){\makebox(0,0)[b]{\scriptsize$2$}}
\put(60,0){\makebox(0,0)[b]{\scriptsize$3$}}
\put(20,35){\makebox(0,0)[b]{\scriptsize$0$}}
\put(40,35){\makebox(0,0)[b]{\scriptsize$1$}}
\end{picture}
\end{center}

For $R=J(f^{op})$ and an appropriate sequence
$\overrightarrow{ab}=2\;2\;\;2\;2\;2\;2$ we have that $F_{\overrightarrow{ab}}(R)$ is
given by

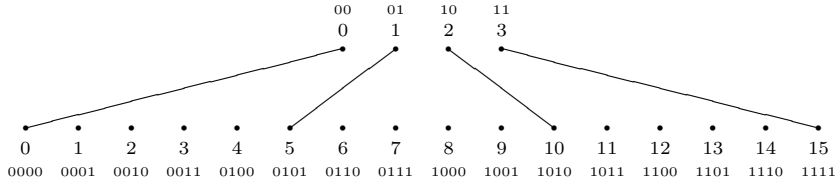
\begin{center}
\begin{picture}(300,60)
\put(0,20){\circle*{2}} \put(20,20){\circle*{2}}
\put(40,20){\circle*{2}} \put(60,20){\circle*{2}}
\put(80,20){\circle*{2}} \put(100,20){\circle*{2}}
\put(120,20){\circle*{2}} \put(140,20){\circle*{2}}
\put(160,20){\circle*{2}} \put(180,20){\circle*{2}}
\put(200,20){\circle*{2}} \put(220,20){\circle*{2}}
\put(240,20){\circle*{2}} \put(260,20){\circle*{2}}
\put(280,20){\circle*{2}} \put(300,20){\circle*{2}}

\put(120,50){\circle*{2}} \put(140,50){\circle*{2}}
\put(160,50){\circle*{2}} \put(180,50){\circle*{2}}

\put(120,50){\line(-4,-1){120}} \put(140,50){\line(-4,-3){40}}
\put(180,50){\line(4,-1){120}} \put(160,50){\line(4,-3){40}}

\put(0,10){\makebox(0,0)[b]{\scriptsize$0$}}
\put(20,10){\makebox(0,0)[b]{\scriptsize$1$}}
\put(40,10){\makebox(0,0)[b]{\scriptsize$2$}}
\put(60,10){\makebox(0,0)[b]{\scriptsize$3$}}
\put(80,10){\makebox(0,0)[b]{\scriptsize$4$}}
\put(100,10){\makebox(0,0)[b]{\scriptsize$5$}}
\put(120,10){\makebox(0,0)[b]{\scriptsize$6$}}
\put(140,10){\makebox(0,0)[b]{\scriptsize$7$}}
\put(160,10){\makebox(0,0)[b]{\scriptsize$8$}}
\put(180,10){\makebox(0,0)[b]{\scriptsize$9$}}
\put(200,10){\makebox(0,0)[b]{\scriptsize$10$}}
\put(220,10){\makebox(0,0)[b]{\scriptsize$11$}}
\put(240,10){\makebox(0,0)[b]{\scriptsize$12$}}
\put(260,10){\makebox(0,0)[b]{\scriptsize$13$}}
\put(280,10){\makebox(0,0)[b]{\scriptsize$14$}}
\put(300,10){\makebox(0,0)[b]{\scriptsize$15$}}

\put(0,2){\makebox(0,0)[b]{\tiny$0000$}}
\put(20,2){\makebox(0,0)[b]{\tiny$0001$}}
\put(40,2){\makebox(0,0)[b]{\tiny$0010$}}
\put(60,2){\makebox(0,0)[b]{\tiny$0011$}}
\put(80,2){\makebox(0,0)[b]{\tiny$0100$}}
\put(100,2){\makebox(0,0)[b]{\tiny$0101$}}
\put(120,2){\makebox(0,0)[b]{\tiny$0110$}}
\put(140,2){\makebox(0,0)[b]{\tiny$0111$}}
\put(160,2){\makebox(0,0)[b]{\tiny$1000$}}
\put(180,2){\makebox(0,0)[b]{\tiny$1001$}}
\put(200,2){\makebox(0,0)[b]{\tiny$1010$}}
\put(220,2){\makebox(0,0)[b]{\tiny$1011$}}
\put(240,2){\makebox(0,0)[b]{\tiny$1100$}}
\put(260,2){\makebox(0,0)[b]{\tiny$1101$}}
\put(280,2){\makebox(0,0)[b]{\tiny$1110$}}
\put(300,2){\makebox(0,0)[b]{\tiny$1111$}}

\put(120,55){\makebox(0,0)[b]{\scriptsize$0$}}
\put(140,55){\makebox(0,0)[b]{\scriptsize$1$}}
\put(160,55){\makebox(0,0)[b]{\scriptsize$2$}}
\put(180,55){\makebox(0,0)[b]{\scriptsize$3$}}

\put(120,63){\makebox(0,0)[b]{\tiny$00$}}
\put(140,63){\makebox(0,0)[b]{\tiny$01$}}
\put(160,63){\makebox(0,0)[b]{\tiny$10$}}
\put(180,63){\makebox(0,0)[b]{\tiny$11$}}
\end{picture}
\end{center}
which for the isomorphism
\[
\iota\!:(2\cdot 2)\times(2\cdot 2)\rightarrow (2\cdot2)\cdot (2\cdot 2)
\]
is equal to the composition $\iota\cirk\Delta$ given by
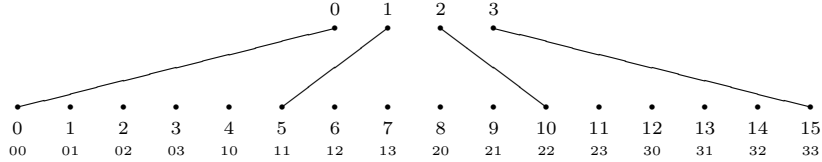
\begin{center}
\begin{picture}(300,60)
\put(0,20){\circle*{2}} \put(20,20){\circle*{2}}
\put(40,20){\circle*{2}} \put(60,20){\circle*{2}}
\put(80,20){\circle*{2}} \put(100,20){\circle*{2}}
\put(120,20){\circle*{2}} \put(140,20){\circle*{2}}
\put(160,20){\circle*{2}} \put(180,20){\circle*{2}}
\put(200,20){\circle*{2}} \put(220,20){\circle*{2}}
\put(240,20){\circle*{2}} \put(260,20){\circle*{2}}
\put(280,20){\circle*{2}} \put(300,20){\circle*{2}}

\put(120,50){\circle*{2}} \put(140,50){\circle*{2}}
\put(160,50){\circle*{2}} \put(180,50){\circle*{2}}

\put(120,50){\line(-4,-1){120}} \put(140,50){\line(-4,-3){40}}
\put(180,50){\line(4,-1){120}} \put(160,50){\line(4,-3){40}}

\put(0,10){\makebox(0,0)[b]{\scriptsize$0$}}
\put(20,10){\makebox(0,0)[b]{\scriptsize$1$}}
\put(40,10){\makebox(0,0)[b]{\scriptsize$2$}}
\put(60,10){\makebox(0,0)[b]{\scriptsize$3$}}
\put(80,10){\makebox(0,0)[b]{\scriptsize$4$}}
\put(100,10){\makebox(0,0)[b]{\scriptsize$5$}}
\put(120,10){\makebox(0,0)[b]{\scriptsize$6$}}
\put(140,10){\makebox(0,0)[b]{\scriptsize$7$}}
\put(160,10){\makebox(0,0)[b]{\scriptsize$8$}}
\put(180,10){\makebox(0,0)[b]{\scriptsize$9$}}
\put(200,10){\makebox(0,0)[b]{\scriptsize$10$}}
\put(220,10){\makebox(0,0)[b]{\scriptsize$11$}}
\put(240,10){\makebox(0,0)[b]{\scriptsize$12$}}
\put(260,10){\makebox(0,0)[b]{\scriptsize$13$}}
\put(280,10){\makebox(0,0)[b]{\scriptsize$14$}}
\put(300,10){\makebox(0,0)[b]{\scriptsize$15$}}

\put(0,2){\makebox(0,0)[b]{\tiny$00$}}
\put(20,2){\makebox(0,0)[b]{\tiny$01$}}
\put(40,2){\makebox(0,0)[b]{\tiny$02$}}
\put(60,2){\makebox(0,0)[b]{\tiny$03$}}
\put(80,2){\makebox(0,0)[b]{\tiny$10$}}
\put(100,2){\makebox(0,0)[b]{\tiny$11$}}
\put(120,2){\makebox(0,0)[b]{\tiny$12$}}
\put(140,2){\makebox(0,0)[b]{\tiny$13$}}
\put(160,2){\makebox(0,0)[b]{\tiny$20$}}
\put(180,2){\makebox(0,0)[b]{\tiny$21$}}
\put(200,2){\makebox(0,0)[b]{\tiny$22$}}
\put(220,2){\makebox(0,0)[b]{\tiny$23$}}
\put(240,2){\makebox(0,0)[b]{\tiny$30$}}
\put(260,2){\makebox(0,0)[b]{\tiny$31$}}
\put(280,2){\makebox(0,0)[b]{\tiny$32$}}
\put(300,2){\makebox(0,0)[b]{\tiny$33$}}

\put(120,55){\makebox(0,0)[b]{\scriptsize$0$}}
\put(140,55){\makebox(0,0)[b]{\scriptsize$1$}}
\put(160,55){\makebox(0,0)[b]{\scriptsize$2$}}
\put(180,55){\makebox(0,0)[b]{\scriptsize$3$}}
\end{picture}
\end{center}
Hence, this is the diagonal arrow from $2\cdot 2$ to $(2\cdot
2)\cdot (2\cdot 2)$ in $\seto$.}
\end{exm}

By reasoning as in Examples \ref{pr1} and \ref{pr2}, we can prove
the following lemma.

\begin{lem}\label{lema6}
The composition $F_{\overrightarrow{ab}}\cirk J$ from $\seto^{op}$ to
$\seto$, for appropriate $\overrightarrow{ab}$, maps projections to
projections and diagonal arrows to diagonal arrows.
\end{lem}

If $a_0=\ldots=b_{m-1}=p\geq 2$, then $F_{\overrightarrow{ab}}(R)$ is denoted
by $F_p(R)$, and it coincides with $F_p(R)$ defined in
\cite[\S5]{DP03a}. In the following lemma we have $F_p$ with
$p=2$.

\begin{lem}\label{lema2}
If $F_{\overrightarrow{ab}}(R)=F_{\overrightarrow{ab}}(S)$, then
$F_2(R)=F_2(S)$.
\end{lem}

\begin{proof}
Let $2^k$ in the index of $\iota^{-1}$ denote the sequence of $k$ occurrences of 2.
For $R,S\!:n\rightarrow m$, for every $(i,j)\in
2^n\times 2^m$, we have that $(i,j)\in F_2(R)$
\begin{tabbing}
\hspace{1.5em}\= iff \= $(\forall x,y\in n\pl m)\; ((x,y)\in
R\Rightarrow
\pi_x(\iota^{-1}_{2^n}(i)\,\iota^{-1}_{2^m}(j))=\pi_y(\iota^{-1}_{2^n}(i)\,\iota^{-1}_{2^m}(j)))$,
\\[1ex]
\> iff \>
$(\iota_{\overrightarrow{a}}(\iota^{-1}_{2^n}(i)),\iota_{\overrightarrow{b}}(\iota^{-1}_{2^m}(j)))
\in F_{\overrightarrow{ab}}(R)$,
\\*[.5ex]
\> \> since
$\iota^{-1}_{2^n}(i)=\iota^{-1}_{\overrightarrow{a}}(\iota_{\overrightarrow{a}}(\iota^{-1}_{2^n}(i)))$
and
$\iota^{-1}_{2^m}(j)=\iota^{-1}_{\overrightarrow{b}}(\iota_{\overrightarrow{b}}(\iota^{-1}_{2^m}(j)))$.
\end{tabbing}

We conclude the same for $R$ replaced by $S$.
\end{proof}

Then, as a corollary of Proposition~5 of \cite[\S5]{DP03a}, we obtain the following lemma.

\begin{lem}\label{lema3}
If $F_2(R)=F_2(S)$, then $R=S$.
\end{lem}

We prove the following lemma along the lines of Proposition~4 of \cite[\S5]{DP03a}.

\begin{lem}\label{lema4}
The function $F_{\overrightarrow{aa}}$ maps the identity to the identity, and
for the arrows ${R\!:n\rightarrow m}$ and $S\!:m\rightarrow p$ of Gen and
appropriate sequences $\overrightarrow{ab}$ and $\overrightarrow{bc}$, we have that
\[F_{\overrightarrow{bc}}(S)\cirk F_{\overrightarrow{ab}}(R)=F_{\overrightarrow{ac}}(S\cirk R).\]
\end{lem}

Let $\mbox{\emph{Gen}}_P$ be the category whose objects are all the
finite sequences of elements of $P$ (as in the category
$\mathcal{C}$) and whose arrows are all the triples of the form
$(R,p_{i_0}\ldots p_{i_{n-1}},p_{j_0}\ldots p_{j_{m-1}})$, where
$R\!:n\rightarrow m$ is an arrow of \emph{Gen} while $p_{i_0}\ldots
p_{i_{n-1}}$ and $p_{j_0}\ldots p_{j_{m-1}}$ are objects of the
category $\mathcal{C}$. From the faithfulness of the composition
$J\cirk G\!:\mathcal{C}\rightarrow\mbox{\emph{Gen}}$, it follows that the
functor $(J\cirk G)_P\!:\mathcal{C}\rightarrow \mbox{\emph{Gen}}_P$,
which is identity on objects, and which maps the arrow
$f\!:p_{i_0}\ldots p_{i_{n-1}}\rightarrow p_{j_0}\ldots p_{j_{m-1}}$ of
$\mathcal{C}$ to $(JGf,p_{i_0}\ldots p_{i_{n-1}},p_{j_0}\ldots
p_{j_{m-1}})$ is also faithful.

Let $\relo$ be the category whose arrows are binary relations
between finite ordinals, and let $F$ be the functor from \emph{Gen} to $\relo$ defined
on objects by
\[
F(p_{i_0}\ldots p_{i_{n-1}})=\mathbf{p}_{i_0}\cdot\ldots\cdot
\mathbf{p}_{i_{n-1}},
\]
and on arrows, for $\overrightarrow{ab}$ the sequence $\mathbf{p}_{i_0}\ldots
\mathbf{p}_{j_{m-1}}$, by
\[
F(R,p_{i_0}\ldots p_{i_{n-1}},p_{j_0}\ldots
p_{j_{m-1}})=F_{\overrightarrow{ab}}(R).
\]
From Lemmata \ref{lema2}-\ref{lema4}, it follows that $F$ is a faithful functor. Hence, the functor $F\cirk(J\cirk
G)_P\!:\mathcal{C}\rightarrow \relo$ is also faithful. By
Lemma \ref{lema0}, we may restrict this functor to a functor from
$\mathcal{C}$ to $\seto$, which we also call $F\cirk(J\cirk G)_P$,
for which we have the following.

\begin{lem}\label{lema5}
$F\cirk(J\cirk G)_P=H$.
\end{lem}
\begin{proof}
The functor $F\cirk(J\cirk G)_P$ maps a generator $p_i$ to the
prime number $\mathbf{p}_i$ and it preserves finite products. By
Lemma \ref{lema6}, it preserves also the rest of the finite
product structure. It remains to apply the uniqueness of $H$ with
these properties.
\end{proof}
Proposition \ref{tvrdjenje1}, which asserts that $H$ is faithful, follows from Lemma \ref{lema5} and
the faithfulness of $F\cirk(J\cirk G)_P$.

Let $sk(\mathcal{C})$ be
a skeleton of $\mathcal{C}$, and let $I$ be the inclusion functor from $sk(\mathcal{C})$ to $\mathcal{C}$. We may consider $sk(\mathcal{C})$
to be the full subcategory of $\mathcal{C}$ on objects of the form $p_{i_1}\ldots p_{i_n}$ with $i_1\leq\ldots\leq i_n$.

\begin{prop}\label{iso}
The composition $H\cirk I\!:sk(\mathcal{C})\rightarrow H\mathcal{C}$ is
an isomorphism.
\end{prop}

\begin{proof}
From Proposition \ref{tvrdjenje1} we have that this composition is
faithful. Since $p_i$ is mapped by $H\cirk I$ to the $i$-th prime
number, and this functor preserves products, it is one-one on
objects.
\end{proof}

The category $sk(\mathcal{C})$ is also a skeleton of the category
with finite products freely generated by $P$. It is free in the
sense that every function from $P$ to the set of objets of a
category from $\catt$ in which product is commutative on objects
extends to a unique arrow of $\catt$ from $sk(\mathcal{C})$ to
this category.

Our categories $\mathcal{D}$ and $sk(\mathcal{C})$ are not exactly
the categories mentioned in the second paragraph of the
introduction. However, they are equivalent to these categories. In
logical terminology, this is just as if we worked with equivalence
classes of formulae instead of usually defined formulae. In the
case of disjunctive formulae, they are identified up to
associativity (we get commutativity for free because we have a
single letter), and in the case of conjunctive formulae, they are
identified up to associativity and commutativity.

By taking opposite categories, one can replace coproducts with
products, disjunction with conjunction and vice versa to obtain
the dual results. Hence, there is nothing asymmetric that gives
priority to disjunction over conjunction in these matters.

\section{Connection with the exponential and contra\-variant power-set functors}
In this concluding section we consider matters that connect our representation of conjunctive deductions by disjunctive deductions with a particular, well-behaved and rather familiar, case of the Brauerian representation of \cite{DP03a}. We consider first the representation of an equivalence relation $R$ by a set of functions ${\cal F}^=(R)$, which we dealt with in \cite[Section~4]{DP03a}, and which engenders the representation by $F_p(R)$, which we dealt with in \cite[Section~5]{DP03a}, and which is closely related to the representation by $F_{\overrightarrow{ab}}(R)$ of this paper (see Section~4 above). The set of functions ${\cal F}^=(R)$ can be replaced by a relation between functions ${\cal F}_{X_1,X_2}(R)$, with $X$ being the disjoint union of $X_1$ and $X_2$. For every equivalence relation $R\subseteq X^2$ there is a function $\Phi\!:X_2\rightarrow X_1$ such that ${\cal F}_{X_1,X_2}(R)$ is equal to the function $p^\Phi$ that maps a function $f_1\!:{X_1\rightarrow p}$ to the function $f_1\cirk\Phi\!:{X_2\rightarrow p}$. Finally, we consider how our representation of conjunctive deductions by disjunctive deductions is related to the exponential functor $p^-$ from \set\ to $\set^{op}$, which on arrows is defined as $p^\Phi$. The exponential functor $2^-$ is naturally isomorphic to the contravariant power-set functor.

For an arbitrary equivalence relation $R\subseteq X^2$ and an arbitrary set $p$ such that for $p_0\neq p_1$ we have $p_0,p_1\in p$, let ${\cal F}^=(R)$ be the set of all functions $f\!:X\rightarrow p$ such that
\begin{tabbing}
\hspace{1.7em}$(\ast)$\hspace{5em}$(\forall x,y\in X)(xRy\Rightarrow f(x)=f(y)).$
\end{tabbing}
It is shown in \cite[Section~4, Corollary]{DP03a} that for $R_1,R_2\subseteq X^2$
equivalence relations we have $R_1= R_2$ iff ${\cal F}^=(R_1)={\cal F}^= (R_2)$.

For an equivalence relation $R\subseteq X^2$, consider the partition of $X$ induced by $R$. Let $X_1$ be a set of representatives of these equivalence classes, one for each class (for the existence of this set one relies on the Axiom of Choice when $X$ is infinite), and let $X_2$ be the complement of $X_1$ with respect to $X$. (The sets $X_1$ and $X_2$ can both be empty and $X_1$ can be nonempty with $X_2$ empty, but $X_1$ cannot be empty with $X_2$ nonempty.) So $X=X_1+X_2$, where $+$ is disjoint union (coproduct in \set).

Let $\Phi\!:X_2\rightarrow X_1$ be the function that maps every element of $X_2$ to the representative of its equivalence class. Assuming the sets $X_1$ and $X_2$ are given, for every function $f\!:X\rightarrow p$ there is a unique pair of functions $(f_1\!:X_1\rightarrow p,f_2\!:X_2\rightarrow p)$ such that $f=[f_1,f_2]$. We can easily verify that $(\ast)$ above is equivalent with
\begin{tabbing}
\hspace{1.7em}$(\ast\ast)$\hspace{5em}$(\forall x_1\in X_1)(\forall x_2\in X_2)(x_1Rx_2\Rightarrow f_1(x_1)=f_2(x_2)).$
\end{tabbing}
We also have that $x_1Rx_2$ iff $\Phi(x_2)=x_1$. From that we infer that $(\ast\ast)$ is equivalent with $f_1\cirk\Phi = f_2$.

Let ${\cal F}_{X_1,X_2}(R)$ be the set of all pairs $(f_1\!:X_1\rightarrow p,f_2\!:X_2\rightarrow p)$ such that $[f_1,f_2]\in{\cal F}^=(R)$. For $\Phi\!:X_2\rightarrow X_1$ as above, let $p^\Phi\!:p^{X_1}\rightarrow p^{X_2}$ be the function that maps a function $f_1\!:X_1\rightarrow p$ to the function $f_1\cirk\Phi\!:X_2\rightarrow p$. We have that
\begin{tabbing}
\hspace{8.5em}$(f_1,f_2)\in{\cal F}_{X_1,X_2}(R)$ \= iff $[f_1,f_2]\in{\cal F}^=(R)$,\\*[.5ex]
\> iff $f_1\cirk\Phi = f_2$,\\[.5ex]
\> iff $p^\Phi(f_1) = f_2$.
\end{tabbing}
So ${\cal F}_{X_1,X_2}(R)$ and $p^\Phi$ are the same function.

For every cartesian closed category ${\cal K}$ (see \cite[Section IV.10]{ML71} and \cite[Section I.3]{LS86}) and every object $C$ of ${\cal K}$ there is an exponential functor $C^-$ from ${\cal K}$ to ${\cal K}^{op}$, which assigns to an object $A$ of ${\cal K}$ the exponential object $C^A$ of ${\cal K}$, and to an arrow $f\!:A\rightarrow B$ of ${\cal K}$ the canonical arrow $C^f\!:C^B\rightarrow C^A$ produced by the cartesian closed structure of ${\cal K}$. (In the notation of \cite[Section I.1]{LS86} the arrow $C^f$ is $\varepsilon_{C,B}\cirk (\mj_{C^B} \times f))^*$.) The category \set\ is cartesian closed, and in it we have the following exponential functor $p^-$ from \set\ to $\set^{op}$, for $p$ an arbitrary set. The set $p^A$ is the set of all functions $h\!:A\rightarrow p$. For $f\!:A\rightarrow B$, the function $p^f\!:p^B\rightarrow p^A$ is defined by taking that for $g\!:B\rightarrow p$ we have
\[
p^f(g) = g\cirk f\!: A\rightarrow p.
\]
The function $p^\Phi$ above is a particular case of $p^f$.

Consider next the contravariant power-set functor of \set\ (see \cite[Section II.2]{ML71}). This is the functor $\overline{P}$ from \set\ to $\set^{op}$ such that $\overline{P}A$ is the power set of the set $A$, and for $f\!:A\rightarrow B$ the function $\overline{P} f\!:\overline{P}B \rightarrow \overline{P}A$ is the inverse-image function under $f$, which means that for $Y\in \overline{P}B$, i.e.\ $Y\subseteq B$, we have
\[
(\overline{P}f)(Y) = \{a\in A\mid f(a)\in Y\}.
\]
It is an easy exercise to verify that the functors $2^-$ and $\overline{P}$ are naturally isomorphic.

The image-function under $f$ and the inverse-image function under $f$ make a covariant Galois connection, i.e.\ a trivial adjunction (see \cite{D98} or \cite[Section 2.4.4]{D99}).

The functors $p^-$ and $\overline{P}$ can be restricted to the category \emph{Finset}, the full subcategory of \set\ whose objects are finite sets. The category $\seto$ is a skeleton of \emph{Finset}. The functor $p^-$ maps coproduct into product, and this fact is related to the arithmetical equation $p^{n+m} = p^n\cdot p^m$.

This paper gave a representation of product through coproduct in $\seto$. This worked because product is tied to functions, and functions in $\seto$ in general are representable through coproduct. This is not specific for functions tied to product. Any other functions would be representable through coproduct in $\seto$. For example, functions tied to exponentiation, which logically corresponds to implication. Since, as we noted at the end of the preceding section, coproduct can be represented through product in $\seto$, anything representable through coproduct can also be represented through product in $\seto$.

\end{document}